\documentclass[12pt]{amsart}

\usepackage[T1]{fontenc}
\usepackage[utf8]{inputenc}
\usepackage{lmodern}
\usepackage[margin=1in]{geometry}
\usepackage{microtype}
\usepackage{amsmath,amssymb,amsthm,mathtools,mathrsfs,eucal}
\usepackage{xcolor}
\usepackage{tikz}
\usetikzlibrary{calc,arrows.meta,positioning}
\usepackage{hyperref}

\hypersetup{
  colorlinks=true,
  linkcolor=blue!50!black,
  citecolor=blue!50!black,
  urlcolor=blue!50!black
}

\newtheorem{theorem}{Theorem}
\newtheorem{lemma}[theorem]{Lemma}
\newtheorem{corollary}[theorem]{Corollary}

\theoremstyle{remark}
\newtheorem{remark}[theorem]{Remark}

\newcommand{\R}{\mathbb{R}}

\DeclareMathOperator{\area}{area}
\DeclareMathOperator{\conv}{conv}
\DeclareMathOperator{\diam}{diam}

\newcommand{\FillRemovedTriangle}[3]{%
  \fill[white]
    ({#1+#3/4},{#2+sqrt(3)*#3/4})
    -- ({#1+3*#3/4},{#2+sqrt(3)*#3/4})
    -- ({#1+#3/2},{#2})
    -- cycle;%
}

\title{How Thick Is the Sierpi\'nski Triangle?}

\author[Scott Duke Kominers]{Scott Duke Kominers}
\address{Harvard Business School; Department of Economics and Center of Mathematical Sciences and Applications, Harvard University; and a16z crypto}
\email{kominers@fas.harvard.edu}

\thanks{I used LLMs to assist with computations and proof-checking in the preparation of this article, particularly GPT-5.$\{\text{4},\text{5}\}$ Pro, and Claude 4.6 Opus (accessed in part via Poe with the support of Quora, where I am an advisor). I particularly appreciate a thorough review from Refine.ink. The problem, methods, and eventual written form are my own; and of course any errors remain my responsibility. This work was conducted while I was visiting the Technological Innovation, Entrepreneurship, and Strategic Management (TIES) Group at the MIT Sloan School of Management; I greatly appreciate their hospitality.}

\begin{document}

\begin{abstract}
Although the Sierpi\'nski triangle has planar area \(0\), it is uniformly non-flat: at every point and every scale, its nearby points span a two-dimensional region of comparable size.
We prove a sharp version of this statement, showing that the Feng--Wu thickness of \(E\) is exactly \(\sqrt{3}/6\), the inradius of a unit equilateral triangle.
More precisely, if \(E\) is the standard Sierpi\'nski triangle of side length \(1\) and \(B(x,r)\) denotes the closed disk of radius \(r\) centered at \(x\), then for every \(x\in E\) and every \(0<r\le 1\), the convex hull of \(E\cap B(x,r)\) contains an equilateral triangle of side length \(r\).
Consequently, \(\conv(E\cap B(x,r))\) contains a closed disk of radius \((\sqrt{3}/6)r\); this constant is best possible.
The proof is elementary---boundary edges of all construction triangles survive in the limit set, and self-similarity reduces the problem to the normalized range \(1/2\le r\le 1\).
\end{abstract}

\maketitle

\section{Introduction}

The Sierpi\'nski triangle is one of the first fractals many of us meet: start with an equilateral triangle, remove the open middle triangle, then repeat the same removal process inside each of the three corner triangles \textit{ad infinitum} \cite{Sierpinski15,Sierpinski16}.
The resulting set has area \(0\)---in a certain sense, it takes up no ``space''---but it still looks strikingly two-dimensional.
(Figure~\ref{fig:sierpinski-finite-stage} shows a finite-stage approximation to the limit set.)
No matter where one zooms in, nearby points do not collapse onto a line; they always spread out in more than one direction.

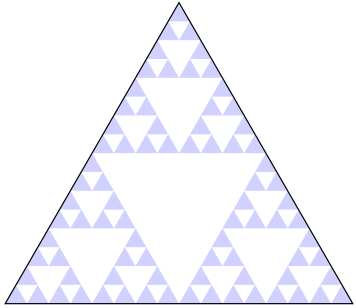
\begin{figure}[t]
\centering
\begin{tikzpicture}[scale=4.6]
  \coordinate (A) at (0,0);
  \coordinate (B) at (1,0);
  \coordinate (C) at (0.5,{sqrt(3)/2});

  \fill[blue!18] (A)--(B)--(C)--cycle;

  % Remove middle triangles through a fixed finite stage.
  % Each triple x/y/s gives the lower-left vertex (x,y) and side length s
  % of an upright construction triangle whose middle triangle is removed.

  \foreach \x/\y/\s in {
    0/0/1
  }{
    \FillRemovedTriangle{\x}{\y}{\s}
  }

  \foreach \x/\y/\s in {
    0/0/0.5,
    0.5/0/0.5,
    0.25/{sqrt(3)/4}/0.5
  }{
    \FillRemovedTriangle{\x}{\y}{\s}
  }

  \foreach \x/\y/\s in {
    0/0/0.25,
    0.25/0/0.25,
    0.125/{sqrt(3)/8}/0.25,
    0.5/0/0.25,
    0.75/0/0.25,
    0.625/{sqrt(3)/8}/0.25,
    0.25/{sqrt(3)/4}/0.25,
    0.5/{sqrt(3)/4}/0.25,
    0.375/{3*sqrt(3)/8}/0.25
  }{
    \FillRemovedTriangle{\x}{\y}{\s}
  }

  \foreach \x/\y/\s in {
    0/0/0.125,
    0.125/0/0.125,
    0.0625/{sqrt(3)/16}/0.125,
    0.25/0/0.125,
    0.375/0/0.125,
    0.3125/{sqrt(3)/16}/0.125,
    0.125/{sqrt(3)/8}/0.125,
    0.25/{sqrt(3)/8}/0.125,
    0.1875/{3*sqrt(3)/16}/0.125,
    0.5/0/0.125,
    0.625/0/0.125,
    0.5625/{sqrt(3)/16}/0.125,
    0.75/0/0.125,
    0.875/0/0.125,
    0.8125/{sqrt(3)/16}/0.125,
    0.625/{sqrt(3)/8}/0.125,
    0.75/{sqrt(3)/8}/0.125,
    0.6875/{3*sqrt(3)/16}/0.125,
    0.25/{sqrt(3)/4}/0.125,
    0.375/{sqrt(3)/4}/0.125,
    0.3125/{5*sqrt(3)/16}/0.125,
    0.5/{sqrt(3)/4}/0.125,
    0.625/{sqrt(3)/4}/0.125,
    0.5625/{5*sqrt(3)/16}/0.125,
    0.375/{3*sqrt(3)/8}/0.125,
    0.5/{3*sqrt(3)/8}/0.125,
    0.4375/{7*sqrt(3)/16}/0.125
  }{
    \FillRemovedTriangle{\x}{\y}{\s}
  }

  \draw[black, line width=0.45pt] (A)--(B)--(C)--cycle;
\end{tikzpicture}
\caption{A finite-stage approximation to the standard Sierpi\'nski triangle. The limit set is obtained by repeating the corner-triangle construction indefinitely.}
\label{fig:sierpinski-finite-stage}
\end{figure}

This note makes that visual intuition precise using the Feng--Wu notion of thickness \cite{FengWu}, which, perhaps a bit paradoxically, distinguishes different ways a set can be thin.
More broadly, we illustrate a useful distinction: a measure-zero set may still occupy surrounding space in a quantitatively substantial geometric way.
A line segment in the plane has area \(0\) because it is genuinely one-dimensional.
The Sierpi\'nski triangle also has area \(0\), but for a very different reason: it branches in multiple directions at every scale while leaving infinitely many holes.
So the statement ``this set has measure zero'' says little by itself about how much of the ambient plane the set fills locally.

A natural way to probe that local geometry for a fractal \(F\) is to fix a point \(x\) and a radius \(r>0\), and look at how much of \(F\) is contained in the ball of radius \(r\) centered at \(x\).
The intersection with the ball may have many holes, but its convex hull records how broadly those intersection points are spread.
If this local convex hull always contains a Euclidean ball whose radius is a definite fraction of \(r\), then there is a sense in which the set \(F\) is uniformly far from being flat: its local convex hulls remain quantitatively two-dimensional at every location and scale.

Feng--Wu thickness \cite{FengWu} packages the idea just described into a numerical invariant.
This invariant was originally introduced to study when arithmetic sums of measure-zero sets become large enough to acquire interior; thus, Feng--Wu thickness measures both a kind of local geometric non-flatness and a kind of additive largeness.

In the case of the Sierpi\'nski triangle, the invariant can be computed exactly.
We show that at every point and every scale, the local convex hull contains an equilateral triangle of the corresponding side length, and hence a disk of the corresponding incircle radius.
Thus the Feng--Wu thickness of the Sierpi\'nski triangle is exactly the inradius of a unit equilateral triangle.

\subsection{The Sierpi\'nski triangle}

We let
\[
v_1=(0,0), \qquad
v_2=(1,0), \qquad
v_3=\Bigl(\frac{1}{2},\frac{\sqrt{3}}{2}\Bigr),
\]
and let
\[
\Delta=\conv\{v_1,v_2,v_3\}.
\]
For \(i=1,2,3\), define
\[
\phi_i(x)=\frac{x+v_i}{2}.
\]
The \emph{standard Sierpi\'nski triangle} is the unique nonempty compact set \(E\subset\R^2\) satisfying
\[
E=\phi_1(E)\cup\phi_2(E)\cup\phi_3(E).
\]
This is the classical self-similar attractor associated to the three corner maps; see Hutchinson \cite{Hutchinson} and Falconer \cite{Falconer} for general background.
We write
\[
T_i=\phi_i(\Delta)\qquad (i=1,2,3)
\]
for the three first-level corner triangles.
Figure~\ref{fig:firststep} recalls the first stage of the construction.

We also make use of the equivalent stage construction. Set \(K_0=\Delta\) and
\[
K_{m+1}=\phi_1(K_m)\cup\phi_2(K_m)\cup\phi_3(K_m)
\qquad (m\ge 0).
\]
Then we have
\[
E=\bigcap_{m\ge 0} K_m.
\]
In particular, \(E\subset\Delta\). Moreover, \(v_1,v_2,v_3\in E\), since each vertex remains in every stage \(K_m\).
This construction also makes the zero-area statement visible: \(K_m\) is a union of \(3^m\) equilateral triangles of side length \(2^{-m}\), so
\[
\area(K_m)
=
\left(\frac{3}{4}\right)^m \area(\Delta)
\longrightarrow 0.
\]
Thus \(E\) has area \(0\).

\begin{figure}[t]
\centering
\begin{tikzpicture}[scale=5.2, every node/.style={font=\small}]
  \coordinate (v1) at (0,0);
  \coordinate (v2) at (1,0);
  \coordinate (v3) at (0.5,{sqrt(3)/2});
  \coordinate (m12) at ($(v1)!0.5!(v2)$);
  \coordinate (m23) at ($(v2)!0.5!(v3)$);
  \coordinate (m13) at ($(v1)!0.5!(v3)$);

  \fill[blue!12] (v1)--(m12)--(m13)--cycle;
  \fill[blue!12] (v2)--(m23)--(m12)--cycle;
  \fill[blue!12] (v3)--(m13)--(m23)--cycle;

  \draw[thick] (v1)--(v2)--(v3)--cycle;
  \draw[thick] (m12)--(m23)--(m13)--cycle;

  \fill (v1) circle (0.45pt) node[below left] {$v_1$};
  \fill (v2) circle (0.45pt) node[below right] {$v_2$};
  \fill (v3) circle (0.45pt) node[above] {$v_3$};

  \node[blue!60!black] at (0.20,0.12) {$T_1$};
  \node[blue!60!black] at (0.80,0.12) {$T_2$};
  \node[blue!60!black] at (0.50,0.60) {$T_3$};
\end{tikzpicture}
\caption{First step of the construction and the cells \(T_1,T_2,T_3\).}
\label{fig:firststep}
\end{figure}
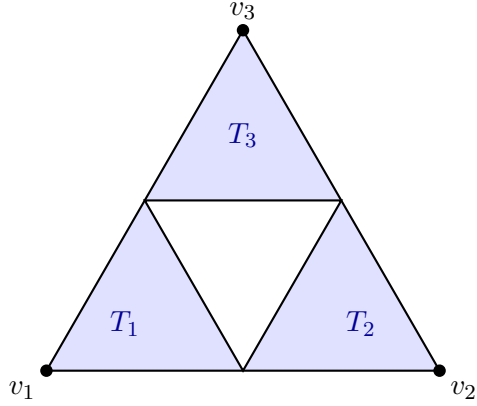

\subsection{Thickness of compact sets in \(\R^2\)}

Throughout the paper, \(B(x,r)\) denotes the closed Euclidean disk of radius \(r\) centered at \(x\).

Following Feng and Wu \cite{FengWu}, we define the \emph{thickness} of a compact set \(F\subset\R^2\) by
\[
\tau(F)=
\sup\Bigl\{
  c\in[0,1]:
  \begin{array}{l}
  \forall x\in F,\ \forall\,0<r\le \diam(F),\\[2pt]
  \exists y\in\R^2 \text{ with } B(y,cr)\subset \conv(F\cap B(x,r))
  \end{array}
\Bigr\}.
\]
This definition is related in spirit, but not in its mechanics, to the classical Newhouse \cite{Newhouse1970,Newhouse1979} concept of thickness for subsets of the line; see~\cite{Yavicoli} for a survey.
Feng--Wu thickness measures the size of balls contained in local convex hulls, rather than ratios of complementary gaps.
It is also distinct from other higher-dimensional notions of thickness used in intersection problems; see, for instance,~\cite{FalconerYavicoli}.

\subsection{Main result}

We compute the Feng--Wu thickness of the Sierpi\'nski triangle exactly.

\begin{theorem}\label{thm:main}
Let \(E\) be the standard Sierpi\'nski triangle of side length \(1\).
Then
\[
\tau(E)=\frac{\sqrt{3}}{6}.
\]
More concretely, for every \(x\in E\) and every \(r\) with \(0<r\le 1\), the convex hull of \(E\cap B(x,r)\) contains an equilateral triangle of side length \(r\), and hence contains a closed disk of radius \(\sqrt{3}\,r/6\); and moreover the constant \(\sqrt{3}/6\) is optimal.
\end{theorem}

The proof comes down to two observations: every boundary edge of every construction triangle survives in the limit set, and once we understand the normalized range \(1/2\le r\le 1\), self-similarity handles all smaller scales.

\section{Boundary edges survive}

For a word \(w=i_1i_2\cdots i_n\) over \(\{1,2,3\}\), write \(|w|=n\).
We allow the empty word \(\varnothing\), with \(|\varnothing|=0\), and set
\[
\phi_{\varnothing}=\operatorname{id},
\qquad
\Delta_{\varnothing}=\Delta.
\]
For \(n\ge 1\), write\footnote{Under our convention, the map on the far right is applied first; this choice is immaterial for the argument, but fixes the notation for the cells \(\Delta_w\).}
\[
\phi_w=\phi_{i_1}\circ\phi_{i_2}\circ\cdots\circ\phi_{i_n},
\qquad
\Delta_w=\phi_w(\Delta).
\]
Thus, when \(|w|=n\), \(\Delta_w\) is a level-\(n\) construction triangle, similar to \(\Delta\) with side length \(2^{-n}\).

The key geometric fact used in our thickness computation is that the boundary of every such triangle is still present in the gasket.

\begin{lemma}\label{lem:boundary}
The three outer sides of \(\Delta\) are contained in \(E\).
In fact, for every word \(w\), the three sides of the construction triangle \(\Delta_w=\phi_w(\Delta)\) are contained in \(E\); in particular, every point on every such side belongs to \(E\).
\end{lemma}

\begin{proof}
Consider first the side \([v_1,v_2]\).
The maps \(\phi_1\) and \(\phi_2\) preserve this side and, in the affine parameter
\[
(1-t)v_1+t v_2 \qquad (0\le t\le 1),
\]
act by
\[
t\mapsto \frac{t}{2}
\qquad\text{and}\qquad
t\mapsto \frac{1+t}{2},
\]
respectively.
Since \(v_1,v_2\in E\) and \(\phi_i(E)\subset E\) for \(i=1,2,3\), applying words in \(\phi_1\) and \(\phi_2\) to the endpoints shows that, for every \(n\ge 0\), all dyadic subdivision points
\[
\left(1-\frac{k}{2^n}\right)v_1+\frac{k}{2^n}v_2
\qquad (k=0,1,\dots,2^n)
\]
belong to \(E\).
These points are dense in \([v_1,v_2]\), and \(E\) is closed.
Therefore \([v_1,v_2]\subset E\).

The same argument applied to the pairs \((\phi_1,\phi_3)\) and \((\phi_2,\phi_3)\) gives
\[
[v_1,v_3]\subset E
\qquad\text{and}\qquad
[v_2,v_3]\subset E.
\]
Thus we see that all three sides of \(\Delta\) are contained in \(E\).

Now let \(w\) be any word over \(\{1,2,3\}\).
Each side of \(\Delta_w=\phi_w(\Delta)\) is the image under \(\phi_w\) of one of the three sides of \(\Delta\).
Moreover, since \(\phi_i(E)\subset E\) for each \(i\), repeated composition gives
\[
\phi_w(E)\subset E.
\]
Applying \(\phi_w\) to any side of \(\Delta\) therefore gives a side of \(\Delta_w\) contained in \(E\).
Thus every side of \(\Delta_w\) lies in \(E\).
\end{proof}

Combining Lemma~\ref{lem:boundary} with the inclusion \(E\subset\Delta\) from the stage construction, we see that \(E\) contains \(v_1,v_2,v_3\) and no point outside \(\Delta\).
Hence
\[
\diam(E)=1
\qquad\text{and}\qquad
\conv(E)=\Delta.
\]

Lemma~\ref{lem:boundary} is the key simplification for what follows: although the construction removes area at every stage, it leaves enough straight boundary segments in \(E\) to build local equilateral triangles by taking convex hulls.

\section{A local triangle in the normalized range}

Recall that \(T_i=\phi_i(\Delta)\) for \(i=1,2,3\).
Each \(T_i\) is an equilateral triangle of side length \(1/2\).

For \(r\in[1/2,1]\), define three larger ``corner triangles:''
\begin{gather*}
Q_1(r)=\conv\!\bigl\{v_1,\ v_1+r(v_2-v_1),\ v_1+r(v_3-v_1)\bigr\},\\
Q_2(r)=\conv\!\bigl\{v_2,\ v_2+r(v_1-v_2),\ v_2+r(v_3-v_2)\bigr\},\\
Q_3(r)=\conv\!\bigl\{v_3,\ v_3+r(v_1-v_3),\ v_3+r(v_2-v_3)\bigr\}.
\end{gather*}
Each \(Q_i(r)\) is an equilateral triangle of side length \(r\) with the same corner vertex as \(T_i\), and \(Q_i(r)\) contains \(T_i\) because \(r\ge 1/2\).
Figure~\ref{fig:normalized} shows the geometry.

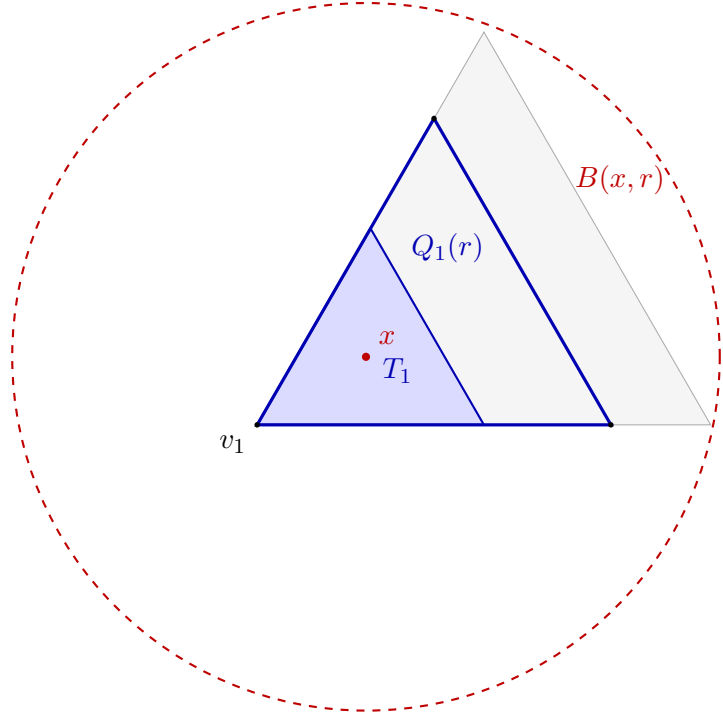
\begin{figure}[t]
\centering
\begin{tikzpicture}[scale=6.0, every node/.style={font=\small}]
  \coordinate (v1) at (0,0);
  \coordinate (v2) at (1,0);
  \coordinate (v3) at (0.5,{sqrt(3)/2});
  \coordinate (m12) at ($(v1)!0.5!(v2)$);
  \coordinate (m13) at ($(v1)!0.5!(v3)$);

  \def\rad{0.78}
  \coordinate (q2) at ($(v1)!\rad!(v2)$);
  \coordinate (q3) at ($(v1)!\rad!(v3)$);
  \coordinate (x) at (0.24,0.15);

  \fill[gray!8] (v1)--(v2)--(v3)--cycle;
  \fill[blue!14] (v1)--(m12)--(m13)--cycle;

  \draw[gray!65] (v1)--(v2)--(v3)--cycle;
  \draw[blue!70!black, very thick] (v1)--(q2)--(q3)--cycle;
  \draw[blue!70!black, thick] (v1)--(m12)--(m13)--cycle;
  \draw[red!75!black, dashed, thick] (x) circle (\rad);

  \fill[red!75!black] (x) circle (0.009) node[above right=1pt] {$x$};
  \fill (v1) circle (0.006) node[below left] {$v_1$};
  \fill (q2) circle (0.006);
  \fill (q3) circle (0.006);

  \node[blue!70!black] at (0.31,0.12) {$T_1$};
  \node[blue!70!black] at (0.42,0.39) {$Q_1(r)$};
  \node[red!75!black] at (0.80,0.54) {$B(x,r)$};
\end{tikzpicture}
\caption{The case \(i=1\): any point \(x\in T_1\) lies inside the corner triangle \(Q_1(r)\), and the three vertices of \(Q_1(r)\) lie in \(B(x,r)\).}
\label{fig:normalized}
\end{figure}

\begin{lemma}\label{lem:onescale}
Let \(x\in E\) and fix \(r\) with \(1/2\le r\le 1\).
Then \(\conv(E\cap B(x,r))\) contains an equilateral triangle of side length \(r\).
\end{lemma}

\begin{proof}
Because
\[
E=\phi_1(E)\cup\phi_2(E)\cup\phi_3(E)\subset T_1\cup T_2\cup T_3,
\]
there is some \(i\in\{1,2,3\}\) with \(x\in T_i\).
The only overlaps between the closed triangles \(T_1,T_2,T_3\) are the three midpoints \(m_{12},m_{13},m_{23}\), so if \(x\) belongs to two cells, the choice of \(i\) is harmless.
Since \(T_i\subset Q_i(r)\), we have \(x\in Q_i(r)\).

By Lemma~\ref{lem:boundary}, the three vertices of \(Q_i(r)\) belong to \(E\), since each lies on one of the sides of the original triangle \(\Delta\).
So it remains to check that these three vertices lie in \(B(x,r)\).

In an equilateral triangle of side length \(r\), the closed disk of radius \(r\) centered at any vertex contains all three vertices; since the disk is convex, it contains the whole triangle.
Applying this to \(Q_i(r)\), we see that every point of \(Q_i(r)\)---and in particular, \(x\)---is within distance at most \(r\) of each vertex of \(Q_i(r)\).

Thus the three vertices of \(Q_i(r)\) belong to \(E\cap B(x,r)\).
Their convex hull is exactly \(Q_i(r)\), an equilateral triangle of side length \(r\).
Therefore
\[
Q_i(r)\subset \conv(E\cap B(x,r)),
\]
as claimed.
\end{proof}

Lemma~\ref{lem:onescale} is the heart of the argument, proving the desired local triangle statement in the normalized range \(1/2\le r\le 1\).
The next section uses self-similarity to transfer that normalized statement to all smaller radii.

\section{Scaling down and the exact constant}

The preceding lemma treats radii comparable to the diameter of the whole gasket.
For a smaller radius, we choose a construction scale whose cell side length is comparable to that radius, take a cell at that scale containing \(x\), pull that cell back to the original gasket, apply the normalized lemma, and then scale the resulting triangle forward again.
Figure~\ref{fig:rescaling} illustrates this rescaling step.

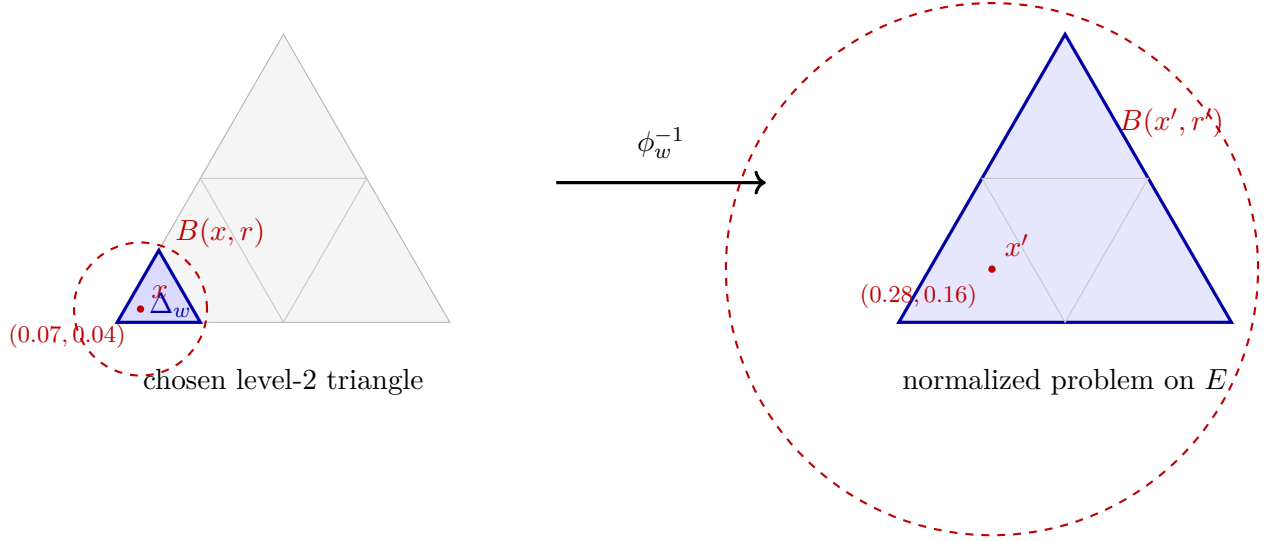
\begin{figure}[t]
\centering
\begin{tikzpicture}[scale=4.4, every node/.style={font=\small}]
  % The displayed example uses the lower-left level-2 cell, corresponding
  % to \phi_w(z)=z/4. Thus n=2, x'=\phi_w^{-1}(x)=4x, and r'=4r.
  \pgfmathsetmacro{\xcoord}{0.07}
  \pgfmathsetmacro{\ycoord}{0.04}
  \pgfmathsetmacro{\rr}{0.20}
  \pgfmathsetmacro{\scaleFactor}{4}
  \pgfmathsetmacro{\xpcoord}{\scaleFactor*\xcoord}
  \pgfmathsetmacro{\ypcoord}{\scaleFactor*\ycoord}
  \pgfmathsetmacro{\rprime}{\scaleFactor*\rr}

  \begin{scope}
    \coordinate (a1) at (0,0);
    \coordinate (a2) at (1,0);
    \coordinate (a3) at (0.5,{sqrt(3)/2});
    \coordinate (b1) at ($(a1)!0.5!(a2)$);
    \coordinate (b2) at ($(a2)!0.5!(a3)$);
    \coordinate (b3) at ($(a1)!0.5!(a3)$);

    \coordinate (c1) at (0,0);
    \coordinate (c2) at ($(a1)!0.25!(a2)$);
    \coordinate (c3) at ($(a1)!0.25!(a3)$);

    \coordinate (x) at (\xcoord,\ycoord);

    \fill[gray!8] (a1)--(a2)--(a3)--cycle;
    \draw[gray!60] (a1)--(a2)--(a3)--cycle;
    \draw[gray!45] (b1)--(b2)--(b3)--cycle;

    \fill[blue!14] (c1)--(c2)--(c3)--cycle;
    \draw[blue!65!black, very thick] (c1)--(c2)--(c3)--cycle;

    \draw[red!75!black, dashed, thick] (x) circle (\rr);

    \fill[red!75!black] (x) circle (0.011);
    \node[red!75!black, above right=0pt] at (x) {\(x\)};
    \node[red!75!black, below left=2pt] at (x) {\scriptsize \((0.07,0.04)\)};

    \node[blue!65!black] at (0.16,0.054) {\(\Delta_w\)};
    \node[red!75!black] at (0.31,0.27) {\(B(x,r)\)};
    \node at (0.5,-0.18) {chosen level-\(2\) triangle};
  \end{scope}

  \draw[->, very thick] (1.32,0.42) -- (1.95,0.42);
  \node at (1.635,0.54) {\(\phi_w^{-1}\)};

  \begin{scope}[shift={(2.35,0)}]
    \coordinate (v1) at (0,0);
    \coordinate (v2) at (1,0);
    \coordinate (v3) at (0.5,{sqrt(3)/2});
    \coordinate (m12) at ($(v1)!0.5!(v2)$);
    \coordinate (m23) at ($(v2)!0.5!(v3)$);
    \coordinate (m13) at ($(v1)!0.5!(v3)$);

    \coordinate (xprime) at (\xpcoord,\ypcoord);

    \fill[blue!10] (v1)--(v2)--(v3)--cycle;
    \draw[blue!65!black, very thick] (v1)--(v2)--(v3)--cycle;

    \draw[gray!45] (m12)--(m23)--(m13)--cycle;
    \draw[red!75!black, dashed, thick] (xprime) circle (\rprime);

    \fill[red!75!black] (xprime) circle (0.011);
    \node[red!75!black, above right=1pt] at (xprime) {\(x'\)};
    \node[red!75!black, below left=2pt] at (xprime) {\scriptsize \((0.28,0.16)\)};

    \node[red!75!black] at (0.82,0.60) {\(B(x',r')\)};
    \node at (0.5,-0.18) {normalized problem on \(E\)};
  \end{scope}
\end{tikzpicture}
\caption{The rescaling step, illustrated for the lower-left level-\(2\) cell, where \(\phi_w(z)=z/4\), so \(x'=\phi_w^{-1}(x)=4x\) and \(r'=4r\). The map \(\phi_w^{-1}\) carries the chosen cell \(\Delta_w\) onto the full normalized triangle \(\Delta\), so \(x'\) has the same relative position in \(\Delta\) that \(x\) has in \(\Delta_w\). In general, for \(|w|=n\), applying \(\phi_w^{-1}\) sends \(x\) to \(x'=\phi_w^{-1}(x)\) and rescales the radius to \(r'=2^n r\in[1/2,1]\), which is the range covered by Lemma~\ref{lem:onescale}.}
\label{fig:rescaling}
\end{figure}

\begin{lemma}\label{lem:allscales}
For every \(x\in E\) and every \(r\) with \(0<r\le 1\), the set \(\conv(E\cap B(x,r))\) contains an equilateral triangle of side length \(r\).
\end{lemma}

\begin{proof}
Choose \(n\ge 0\) such that
\[
2^{-(n+1)}\le r\le 2^{-n},
\]
or equivalently
\[
\frac{1}{2}\le 2^n r\le 1.
\]

Iterating the identity \(E=\phi_1(E)\cup\phi_2(E)\cup\phi_3(E)\) gives
\[
E=\bigcup_{|w|=n}\phi_w(E).
\]
Choose a word \(w\) of length \(n\) such that \(x\in\phi_w(E)\), and write
\[
x=\phi_w(x')
\qquad\text{with \(x'\in E\)}.
\]
Set \(r'=2^n r\), so \(r'\in[1/2,1]\).

By Lemma~\ref{lem:onescale}, the set \(\conv(E\cap B(x',r'))\) contains an equilateral triangle \(Q'\) of side length \(r'\).
Applying \(\phi_w\), we obtain an equilateral triangle
\[
Q=\phi_w(Q')
\]
of side length \(2^{-n}r'=r\).

If \(p'\in E\cap B(x',r')\), then \(\phi_w(p')\in E\) and
\[
|\phi_w(p')-x|=2^{-n}|p'-x'|\le 2^{-n}r'=r,
\]
so \(\phi_w(p')\in E\cap B(x,r)\).
Hence
\[
\phi_w(E\cap B(x',r'))\subset E\cap B(x,r).
\]
Since \(\phi_w\) is affine,
\[
Q
\subset
\phi_w\bigl(\conv(E\cap B(x',r'))\bigr)
=
\conv\bigl(\phi_w(E\cap B(x',r'))\bigr)
\subset
\conv(E\cap B(x,r));
\]
this proves the claim.
\end{proof}

It remains to translate Lemma~\ref{lem:allscales} into the numerical value of \(\tau(E)\) and to show that the resulting constant is best possible.

\begin{proof}[Proof of Theorem~\ref{thm:main}]
Fix \(x\in E\) and \(r\) with \(0<r\le 1=\diam(E)\).
By Lemma~\ref{lem:allscales}, the set \(\conv(E\cap B(x,r))\) contains an equilateral triangle of side length \(r\).
The inscribed closed disk of such a triangle---that is, its incircle---has radius
\[
\frac{\sqrt{3}}{6}\,r.
\]

We thus know that for all \(x\in E\) and all \(r\) with \(0<r\le \diam(E)\), there exists \(y\in\R^2\) such that
\[
B\left(y,\frac{\sqrt{3}}{6}r\right)\subset \conv(E\cap B(x,r)).
\]
Recalling the definition of Feng--Wu thickness, this shows that
\[
\tau(E)\ge \frac{\sqrt{3}}{6}.
\]

For the reverse inequality, take \(x=v_1\) and \(r=1\).
Since \(E\subset\Delta\subset B(v_1,1)\), we have
\[
E\cap B(v_1,1)=E.
\]
Also \(\conv(E)=\Delta\).
Thus, for this particular choice of \(x\) and \(r\), the local convex hull in the definition of thickness is exactly the unit equilateral triangle \(\Delta\).

Meanwhile, no disk contained in \(\Delta\) can have radius larger than \(\sqrt{3}/6\).
To see this, let \(D\) be a disk of radius \(\rho\) contained in \(\Delta\), and let \(z\) be its center.
Since \(\Delta\) is convex, \(z\in\Delta\).
Let \(d_1,d_2,d_3\) be the perpendicular distances from \(z\) to the three sides of \(\Delta\).
Since \(D\subset\Delta\), we have
\[
\rho\le \min\{d_1,d_2,d_3\}.
\]
On the other hand, decomposing \(\Delta\) into the three triangles with common vertex \(z\) and bases the three sides of \(\Delta\), each of length \(1\), gives
\[
\area(\Delta)=\frac{1}{2}(d_1+d_2+d_3).
\]
Because \(\Delta\) has side length \(1\) and altitude \(\sqrt{3}/2\), also
\[
\area(\Delta)=\frac{1}{2}\cdot\frac{\sqrt{3}}{2}.
\]
Thus
\[
d_1+d_2+d_3=\frac{\sqrt{3}}{2},
\]
and hence
\[
\rho\le \min\{d_1,d_2,d_3\}\le \frac{\sqrt{3}}{6};
\]
equality is attained by the incircle.
Therefore, for \(x=v_1\) and \(r=1\), the defining condition for \(\tau(\cdot)\) fails for every \(c>\sqrt{3}/6\).
Hence
\[
\tau(E)\le \frac{\sqrt{3}}{6}.
\]

Combining the two inequalities gives
\[
\tau(E)=\frac{\sqrt{3}}{6}.\qedhere
\]
\end{proof}

\begin{remark}
The upper bound is already apparent at the largest possible scale and at a corner point: when \(x=v_1\) and \(r=1\), the local convex hull is simply the ambient triangle \(\Delta\).
So the exact thickness is forced by the most intuitively visible layer of the geometry.
\end{remark}

\section{A sumset remark}

One reason to compute Feng--Wu thickness explicitly is that it gives quantitative criteria for when sums of thin sets acquire interior.

Recall that the Minkowski sum of sets \(A,B\subset\R^2\) is
\[
A+B=\{a+b:a\in A,\ b\in B\}.
\]
A recent theorem of~\cite{Kominers}, sharpening a result of~\cite{FengWu}, states that if compact sets \(F_1,\dots,F_n\subset\R^d\) have positive diameter and common Feng--Wu thickness at least \(c>0\), then
\[
F_1+\cdots+F_n
\]
has nonempty interior whenever
\[
n>\frac{\sqrt d}{(\sqrt{1+c}-1)^2}.
\]
Thus Theorem~\ref{thm:main} immediately gives an explicit many-summand interior result.
This consequence is not intended to be sharp for the Sierpi\'nski triangle; rather, it illustrates how the local convex-geometric invariant enters general sumset theorems.
Applying the theorem with \(F_1=\cdots=F_n=E\), \(d=2\), and \(c=\sqrt{3}/6\), we obtain the following.

\begin{corollary}\label{cor:manysum}
If
\[
n>\frac{\sqrt2}{\bigl(\sqrt{1+\sqrt{3}/6}-1\bigr)^2}\approx 77.37,
\]
then the \(n\)-fold sumset
\[
nE=\underbrace{E+\cdots+E}_{\text{\(n\) times}}
\]
has nonempty interior.
In particular, \(nE\) has nonempty interior for every integer \(n\ge 78\).
\end{corollary}

Corollary~\ref{cor:manysum} is a black-box consequence of the general theorem, and it is not sharp for the Sierpi\'nski gasket.
The special boundary geometry of \(E\) gives interior much sooner: already the two-fold sumset has interior.

\begin{corollary}\label{cor:twosum}
The sumset \(E+E\) has nonempty interior.
\end{corollary}

Figure~\ref{fig:twosum} illustrates the geometry behind the argument: two surviving boundary segments already generate a two-dimensional parallelogram inside \(E+E\).

\begin{figure}[t]
\centering
\begin{tikzpicture}[
    scale=3.2,
    every node/.style={font=\small},
    point/.style={circle, fill=black, inner sep=1.15pt},
    segblue/.style={blue!70!black, line width=1.25pt},
    segred/.style={red!70!black, line width=1.25pt},
    sumgreen/.style={green!35!black, line width=0.8pt},
    >=Stealth
  ]

  \begin{scope}
    \coordinate (v1) at (0,0);
    \coordinate (v2) at (1,0);
    \coordinate (v3) at (0.5,{sqrt(3)/2});

    \fill[gray!4] (v1)--(v2)--(v3)--cycle;
    \draw[gray!45, line width=0.45pt] (v1)--(v2)--(v3)--cycle;

    \draw[segblue] (v1)--(v2);
    \draw[segred]  (v1)--(v3);

    \node[point] at (v1) {};
    \node[point] at (v2) {};
    \node[point] at (v3) {};

    \node[below left=1pt]  at (v1) {$v_1$};
    \node[below right=1pt] at (v2) {$v_2$};
    \node[above=2pt]       at (v3) {$v_3$};

    \node[blue!70!black, below=7pt] at ($(v1)!0.5!(v2)$)
      {$[v_1,v_2]$};
    \node[red!70!black, rotate=60, above=5pt] at ($(v1)!0.54!(v3)$)
      {$[v_1,v_3]$};

    \node[font=\small] at (0.5,-0.38) {two segments in \(E\)};
  \end{scope}

  \draw[->, line width=0.55pt] (1.42,0.42) -- (2.02,0.42);
  \node[font=\normalsize] at (1.72,0.63) {\(+\)};

  \begin{scope}[shift={(2.42,0)}]
    \coordinate (p11) at (0,0);
    \coordinate (p12) at (1,0);
    \coordinate (p13) at (0.5,{sqrt(3)/2});
    \coordinate (p23) at (1.5,{sqrt(3)/2});

    \fill[green!16] (p11)--(p12)--(p23)--(p13)--cycle;
    \draw[sumgreen] (p11)--(p12)--(p23)--(p13)--cycle;

    \node[point] at (p11) {};
    \node[point] at (p12) {};
    \node[point] at (p13) {};
    \node[point] at (p23) {};

    \node[below left=1pt]  at (p11) {$2v_1$};
    \node[below right=1pt] at (p12) {$v_1+v_2$};
    \node[above left=1pt]  at (p13) {$v_1+v_3$};
    \node[above right=1pt] at (p23) {$v_2+v_3$};

    \node[green!35!black] at (0.75,{sqrt(3)/3})
      {$[v_1,v_2]+[v_1,v_3]$};

    \node[font=\small] at (0.75,-0.38) {a parallelogram in \(E+E\)};
  \end{scope}
\end{tikzpicture}
\caption{The Minkowski sum of two segments contained in \(E\) is a parallelogram contained in \(E+E\).}
\label{fig:twosum}
\end{figure}
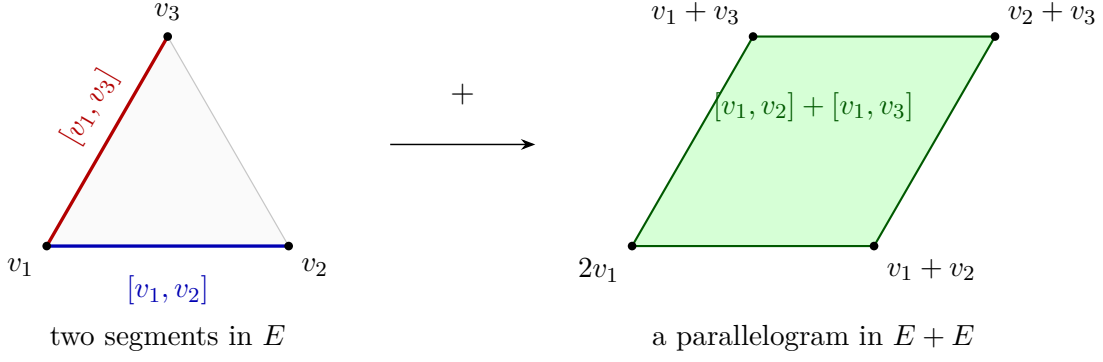

\begin{proof}[Proof of Corollary~\ref{cor:twosum}]
By Lemma~\ref{lem:boundary}, the sides \([v_1,v_2]\) and \([v_1,v_3]\) are contained in \(E\).
Hence
\[
[v_1,v_2]+[v_1,v_3]\subset E+E.
\]
But
\[
[v_1,v_2]+[v_1,v_3]
=
2v_1+\{\,s(v_2-v_1)+t(v_3-v_1):0\le s,t\le 1\,\},
\]
which is a parallelogram with nonempty interior, since \(v_2-v_1\) and \(v_3-v_1\) are not parallel.
Therefore \(E+E\) has nonempty interior.
\end{proof}

Thus the thickness computation fits neatly into the Feng--Wu framework: a local convex-geometric invariant yields an explicit interior statement for arithmetic sums.
In this case, however, Corollary~\ref{cor:twosum} shows that the general bound is far from sharp: the boundary geometry of the gasket is already strong enough to force interior in \(E+E\).

\providecommand{\bysame}{\leavevmode\hbox to3em{\hrulefill}\thinspace}
\providecommand{\MR}{\relax\ifhmode\unskip\space\fi MR }
\providecommand{\MRhref}[2]{%
  \href{http://www.ams.org/mathscinet-getitem?mr=#1}{#2}
}
\providecommand{\href}[2]{#2}

\end{document}